\documentclass[reqno, oneside, 12pt]{amsart}

\usepackage[letterpaper]{geometry}
\usepackage{enumerate, hyperref,url,amssymb,amsmath,amsthm,amsxtra,mathtools,mathrsfs,calc,nccmath,color}

\usepackage{calc}
\usepackage{graphicx}
\usepackage{caption}
\usepackage{subcaption}

\newcommand{\Z}{\mathbb{Z}}
\newcommand{\Q}{\mathbb{Q}}

\newcommand{\ep}{\varepsilon}

\let\temp\phi
\let\phi\varphi
\let\varphi\temp

\renewcommand{\(}{\left(}
\renewcommand{\)}{\right)}

\newcommand{\ord}{\operatorname{ord}}
\newcommand{\SL}{\operatorname{SL}}
\newcommand{\GL}{\operatorname{GL}}

\newcommand{\Tr}{\operatorname{Tr}}

\renewcommand{\sl}{\big|}
\newcommand{\sk}{\big|_k }

\newcommand{\floor}[1]{\left\lfloor #1 \right\rfloor}

\newtheorem{theorem}{Theorem}[section]

\newtheorem{corollary}[theorem]{Corollary}

\theoremstyle{remark}
\newtheorem*{remark}{Remark}

\newtheorem*{example}{Example}

\numberwithin{equation}{section}

\begin{document}

%%%%%%%%%%%%%%%%%%%%%%%%%%%%%%%%%%%%%%%%%
%               Title, Etc.             %
%%%%%%%%%%%%%%%%%%%%%%%%%%%%%%%%%%%%%%%%%

\title{A Higher Weight Analogue of Ogg's Theorem on Weierstrass Points }

\date{\today}
\author{Robert Dicks}
\address{Department of Mathematics\\
University of Illinois\\
Urbana, IL 61801} 
\email{rdicks2@illinois.edu}

%%%%%%%%%%%%%%%%%%%%%%%%%%%%%%%
%          Abstract           %
%%%%%%%%%%%%%%%%%%%%%%%%%%%%%%%

%%%%%%%%%%%%%%%%%%%%%%%%%%%%%%%%%%%%%%%%%
%               Document Text           %
%%%%%%%%%%%%%%%%%%%%%%%%%%%%%%%%%%%%%%%%%

%%%%%%%%%%%%%%%%%%%%%%%%%%%%%%%%%%%%%%%%%%%%%%%%%%%%%%%%%%%%%%%%%%%%%%%%%%%%%%%%
% Introduction
%%%%%%%%%%%%%%%%%%%%%%%%%%%%%%%%%%%%%%%%%%%%%%%%%%%%%%%%%%%%%%%%%%%%%%%%%%%%%%%%
 \begin{abstract}
 For a positive integer $N$, we say that $\infty$ is a Weierstrass point on the modular curve $X_0(N)$ if there is a non-zero cusp form of weight $2$ on $\Gamma_0(N)$ which vanishes at $\infty$ to order greater than the genus of $X_0(N)$. If $p$ is a prime with $p \nmid N$, Ogg proved that $\infty $ is not a Weierstrass point on $X_0(pN)$ if the genus of $X_0(N)$ is $0$. We prove a similar result for even weights $k \geq 4$. We also study the space of weight $k$ cusp forms on $\Gamma_0(N)$ vanishing to order greater than the dimension.
 \end{abstract}
\maketitle
 \section{Introduction}
 
 If $k$ and $N$ are positive integers, let $S_k(N)$ be the rational vector space of cusp forms of weight $k$ on $\Gamma_0(N)$ with rational Fourier coefficients.
   These forms have a Fourier expansion at $\infty$ of the form
\[
  f(z)= \sum_{n=n_0}^{\infty} a(n) q^{n} \text{ }\text{ with }\text{ }a(n_0) \neq 0,
\]
and we define $\ord_{\infty}(f) := n_0$. Let $g(N)=\dim(S_2(N))$ be the genus of $X_0(N)$. We say that $\infty$ is a Weierstrass point on the modular curve $X_0(N)$ if there exists $0 \neq f \in S_2(N)$ such that $\ord_{\infty}(f) > g(N)$. Ogg \cite{Ogg} proved the following theorem.
\begin{theorem}\label{thm:Ogg}
  If $p$ is a prime such that $p \nmid N$, and if $g(N)=0$, then $ \infty $ is not a Weierstrass point on $X_0(pN)$.
\end{theorem}
   A non-geometric proof of Theorem 1.1 was given in [AMR09] (previously, certain cases of level $p\ell$ for distinct primes  $p$ and $\ell$ were considered in \cite{Kohnen}, \cite{Kilger}). To state our first result, when $N$ is a positive integer and $p$ is a prime such that $p \nmid N$, we require the Atkin-Lehner operator $W_{p}^{pN}$ on  $S_{k}(pN)$ defined in \eqref{eqn:Atkin-Lehner}. Furthermore, if  
\[
f(z)= \sum_{n=n_0}^{\infty} a(n) q^{n} \in S_k(N),
\] 
define 
\[
 v_p(f):=\inf\{v_p(a(n))\}.
\]
 With this notation, we prove the following theorem.
\begin{theorem}\label{thm:main}
    Let $N$ be a positive integer and $k$ be a positive even integer. Let 
    $p$ be a prime with $p \geq \max(5,k+1)$ and $p \nmid N$. Suppose that $0 \neq f \in S_k(pN)$ satisfies 
\[
 v_p(f)=0, \text{   }\text{  }v_p(f|_{k}W^{pN}_{p}) \geq 1-k/2.
\]
 Then  
\[
 \ord_{\infty}(f) \leq \dim(S_k(pN)).
\]
\end{theorem}         
   As a corollary, we prove an analogue of Ogg's theorem.
\begin{corollary}\label{thm:analogue}
   Suppose that $N$ is a positive integer, that $k$ is a positive even integer, and that $p$ is a prime with $p \geq \max(5,k+1)$ and $p \nmid N$. If $S_{k}(N)= \{0\}$, then for $0 \neq f \in S_{k}(pN)$, we have 
\[
    \ord_{\infty}(f) \leq \dim(S_k(pN)).
\]
\end{corollary}
    There is a finite list of $N$ and $k$ for which $S_k(N)=\{0\}$. For $k=2$, there are $15$ such values of $N$ \cite[pg. 110]{Ono}. For $k \geq 4$, the rest are
\begin{align*}
    k&=4: N= 1, 2, 3, 4 \\
     k&=6: N= 1, 2 \\
    k&=8,10,14: N= 1. 
\end{align*}
     
     It is natural to seek to understand the subspace of forms $f \in S_{k}(N)$ which vanish to order greater than the dimension.
     If $N$ is a positive integer and $k$ is a positive even integer, define the subspace
\[
      W_{k}(N):=\{f \in S_k(N): \ord_{\infty}(f) > \dim(S_k(N))\}.
\]
     With this notation, we have $W_{2}(N)=\{0\}$ if and only if $\infty$ is not a Weierstrass point on $X_0(N)$. As a corollary of Theorem~\ref{thm:main}, we obtain a bound for $\dim(W_{k}(pN))$.
\begin{corollary}\label{thm:Subspace}
 Suppose that $p \geq \max(5,k+1)$ is a prime satisfying $p \nmid N$. Then we have 
\[
\dim(W_{k}(pN)) \leq \dim(S_k(N)).
\]
 \end{corollary}
Note that this implies Theorem~\ref{thm:Ogg} in the case $k=2$. It is interesting to note that the bound in Corollary~\ref{thm:Subspace} is independent of $p$. Thus, for fixed $N$, the spaces $W_{k}(pN)$ have uniformly bounded dimension as $p \rightarrow \infty$.
\begin{remark}
For squarefree $N$, Arakawa and B\"{o}cherer \cite{Arakawa-Bocherer} study the space
\[
S_k(N)^{*}:= \{ f \in S_k(N):  f\sk W_{p}^{pN}+p^{1-\frac{k}{2}}\sl U_p=0 \ \  \text{for all $p \mid N$} \}.
\]
We will use a similar subspace to prove Corollary~\ref{thm:Subspace}.
\end{remark}
The following examples, which we computed with Magma, illustrate Corollary~\ref{thm:Subspace} for small values of $N$.
\begin{example}
 For an example which is sharp, set $N=1$, $p=19$, and $k=16$. Here, we have $\dim(S_k(pN))=24$ and $\dim(S_k(N))=1$. In this case, there is a form $f \in S_k(pN)$ with $f=q^{25}+\cdots$.
 \end{example}
 \begin{example}
 To get an example which is sharp and for which $pN$ is not prime, set $N=2$, $p=23$, and $k=12$. Here, $\dim(S_k(pN))=64$ and $\dim(S_k(N))=2$. In this case, there are forms $f$ and $g$ with $f=q^{67}+\cdots$ and $g=q^{68}+\cdots$.
 \end{example}
 \begin{example}
 Corollary~\ref{thm:Subspace} is not always sharp. For example, set $N=1$, $p=29$, and $k=28$. Here, $\dim(S_k(pN))=67$ and $\dim(S_k(N))=3$. In this case, there is no non-zero $f \in S_k(N)$ satisfying $\ord(f) > 67$.
 \end{example}
  The paper is organized as follows. Section $2$ contains the background necessary to prove these results. Section $3$ contains the proof of Theorem~\ref{thm:main}, which uses results from \cite{Ahlgren-Masri-Rouse}. Finally, Section $4$ contains the proofs of Corollary~\ref{thm:analogue} and Corollary~\ref{thm:Subspace}. 
\section{Preliminaries on Modular Forms}
The definitions and facts given here can be found in \cite{Diamond-Shurman} and \cite{Ahlgren-Masri-Rouse}. Let $N$ and $k$ be positive integers.
Let $\ep_{\infty}(N)$ denote the number of cusps on $X_0(N)$, let $g(N)$ denote its genus, and let $\ep_{2}(N)$, $\ep_{3}(N)$ denote the numbers of elliptic points of orders $2$ and $3$, respectively. Then we have
\begin{equation}\label{genus}
g(N)=\mfrac{[\SL_2(\Z):\Gamma_0(N)]}{12}-\mfrac12\ep_{\infty}(N)-\mfrac14\ep_{2}(N)-\mfrac13\ep_{3}(N)+1,
\end{equation}
\[ \ep_{2}(N)= \begin{cases} 
0 & \text{if } 4 \mid N, \\
\displaystyle\prod_{p \mid N} \(1+\(\mfrac{-4}{p}\)\) & \text{otherwise,} 
\end{cases}
\] 
\[ \ep_{3}(N)= \begin{cases} 
0 & \text{if } 9 \mid N, \\
\displaystyle\prod_{p \mid N} \(1+\(\mfrac{-3}{p}\)\) & \text{otherwise.} 
\end{cases}
\]
We have the well-known formula
\[
[\SL_{2}(\Z):\Gamma_0(N)]=N\prod_{p \mid N}\(1+\mfrac{1}{p}\).
\] 
For weights $k \geq 4$, we have
\begin{equation}\label{*}
\dim(S_k(N))=(k-1)(g(N)-1)+\floor{\mfrac{k}{4}}\ep_{2}(N)+\floor{\mfrac{k}{3}}\ep_{3}(N)+\(\mfrac{k}{2}-1\)\ep_{\infty}(N).
\end{equation}  
   
     A form in $S_k(N)$ may have forced vanishing at the elliptic points. As in \cite{Ahlgren-Masri-Rouse}, let $\alpha_{2}(N,k)$ and $\alpha_{3}(N,k)$ count the number of forced complex zeroes of a form $f \in S_k(N)$ at the elliptic points of order $2$ and $3$, respectively. These are given by 
\begin{equation}\label{eq:table}
(\alpha_{2}(N,k),\alpha_{3}(N,k))= \begin{cases} 
(\ep_{2}(N),2\ep_{3}(N)) & \text{if } k \equiv 2 \pmod{12}, \\
(0,\ep_{3}(N)) & \text{if } k \equiv 4 \pmod{12}, \\
(\ep_{2}(N),0) & \text{if } k \equiv 6 \pmod{12}, \\
(0,2\ep_{3}(N)) & \text{if } k \equiv 8 \pmod{12}, \\
(\ep_{2}(N),\ep_{3}(N)) & \text{if } k \equiv 10 \pmod{12}, \\
(0,0) & \text{if } k \equiv 0 \pmod{12}.
\end{cases}
\end{equation}  

  If  $d=\dim(S_k(N)),$ then $S_k(N)$ has a basis $\{f_1,...,f_d\}$ with integer coefficients with the property 
\begin{equation}\label{eqn:INTEGRALITY}
f_i(z)=a_iq^{c_i}+O(q^{c_{i}+1}),  \ \ \ \ \ \ \ \ \ \ \ \ \ \ \ \ \ 1 \leq i \leq d,
\end{equation}
where $a_i \neq 0$ and $c_1<c_2<...<c_d.$ This fact implies that every non-zero $f \in S_k(N)$ has bounded denominators.

For $f \in S_k(N)$ and 
\[
 \alpha=\left(\begin{matrix}a & b \\c & d\end{matrix}\right) \in \GL_{2}^{+}(\Q),
\]
 define the weight $k$ slash operator by
\[
f(z)\sk \alpha := \det(\alpha)^{\frac{k}{2}}(cz+d)^{-k}f\(\frac{az+b}{cz+d}\).
\]
 If $p$ is a prime with $p \nmid N$, let $a,b \in \Z$ satisfy $p^{2}a-pNb=p$. Define the Atkin-Lehner operator $W_{p}^{pN}$ on $S_k(pN)$ by 
   \begin{equation}\label{eqn:Atkin-Lehner}
   f\sk W_{p}^{pN}:= f\sk\left(\begin{matrix}pa & 1 \\pNb & p\end{matrix}\right).
   \end{equation}
The operator $W_{p}^{pN}$ preserves the rationality of the coefficients of $f \in S_k(N)$
\cite[~Thm.~2.6]{rationality}.  
For any prime $p$, define the $U_p$ operator by
   \[
   \(\sum a(n) q^{n}\) \mid U_{p}:= \sum a(pn)q^{n}.
      \]
The trace map
 \[
   \Tr_{N}^{pN}:S_k(pN) \rightarrow S_k(N)
   \]
is defined by 
\[
\Tr_{N}^{pN}(f):=f+p^{1-\tfrac{k}{2}}f\sk W_{p}^{pN}\sl U_{p}.
\]
This map is surjective, since for $f \in S_{k}(N)$, we have $\Tr_{N}^{pN}(f)=(p+1)f$.
      \section{Proof of Theorem~\ref{thm:main}}
   
    Let $N$ be a positive integer and $k$ be a positive even integer. When $k=2$, Theorem~\ref{thm:main} follows from \cite[Thm. 1.1]{Ahlgren-Masri-Rouse}.
    Therefore, we may assume that $k \geq 4$. Throughout, let 
\[
\alpha_2:=\alpha_2(N,(k-1)p+1),
\] 
\[
\alpha_3:= \alpha_3(N,(k-1)p+1),
\]
and
\[
I(N):=[\SL_2(\Z):\Gamma_0(N)].
\]
    Suppose that $p \geq \max(5,k+1)$ is a prime with $p \nmid N$, and that $f \in S_{k}(pN)$ satisfies $v_p(f)=0$ and $v_p(f\sk W_{p}^{pN}) \geq 1-\frac{k}{2}$. By \cite[Thm. 4.2]{Ahlgren-Masri-Rouse}, we have
\begin{equation}
 \ord_{\infty}(f) \leq \mfrac{(k-1)p+1}{12}I(N)-\mfrac{1}{2}\alpha_{2}-\mfrac{1}{3}\alpha_{3}-\ep_{\infty}(N)+1.
 \end{equation}
 Using \eqref{genus}, \eqref{*}, and the facts that $I(pN)=(p+1)I(N)$ and $\ep_{\infty}(pN)=2\ep_{\infty}(N)$,
 the proof of Theorem~\ref{thm:main} reduces to proving that
 \begin{equation}\label{eqn:main ineq}
 \mfrac{k-2}{12}I(N) +\(\floor{\mfrac{k}{4}}  - \mfrac{k-1}{4}\)\ep_{2}(pN)+\(\floor{\mfrac{k}{3}}-\mfrac{k-1}{3}\)\ep_{3}(pN)+\mfrac{1}{2}\alpha_{2}+\mfrac{1}{3}\alpha_{3} \geq 1.
 \end{equation}
 The proof of \eqref{eqn:main ineq} breaks up into several cases.  
 \subsection{$\alpha_{2}=0 \text{ and } \alpha_{3}=0$}
 Suppose that $\ep_{2}(N)=\ep_{3}(N)=0$.
Then \eqref{eqn:main ineq}  simplifies to
\begin{equation}\label{eqn:simple1}
 \mfrac{k-2}{12}I(N) \geq 1.
 \end{equation}
 The definitions of $\ep_{2}(N)$ and $\ep_{3}(N)$ imply that $N \geq 4$. Thus, we have \eqref{eqn:main ineq} because $I(N) \geq 6$ whenever $N \geq 4$.
  
  Assume now that $\ep_{2}(N) \neq 0 \text{ and } \ep_{3}(N)=0$.
 From \eqref{eq:table}, we have 
 \[
 (k-1)p+1 \equiv 0 \pmod{4}, 
\]
so
\[
(k,p) \equiv (0,1) \text{ or } (2,3) \pmod{4}.
\]
The definitions of $\ep_{2}(N)$ and $\ep_{3}(N)$ imply that $N \geq 2$, so that $I(N) \geq 3$. In the former case, \eqref{eqn:main ineq} reduces to 
\begin{equation}\label{eqn:simple3}
\mfrac{k-2}{12}I(N)+\mfrac{1}{2}\ep_{2}(N) \geq 1,
\end{equation}
which holds since $k \geq 4$. In the latter case, \eqref{eqn:main ineq} reduces to \eqref{eqn:simple1}, which holds since $k \geq 6$.

  Now assume that $\ep_{2}(N)=0 \text{ and } \ep_{3}(N) \neq 0$. From \eqref{eq:table}, we have $(k-1)p+1 \equiv 0 \pmod{3}$, so
\[
(k,p) \equiv (0,1) \text{ or } (2,2) \pmod{3}.
\]
   In the first case, \eqref{eqn:main ineq} reduces to
\begin{equation}\label{eqn:2}
\mfrac{k-2}{12}I(N) +\mfrac{2}{3}\ep_{3}(N)\geq 1,
\end{equation}
which holds since $k \geq 6$.
If $k \equiv 2 \pmod{3}$, then \eqref{eqn:main ineq} reduces to \eqref{eqn:simple1},
which holds since $k \geq 8$.
 
 Finally, assume that $\ep_{2}(N) \neq 0$ and $\ep_{3}(N) \neq 0$. By \eqref{eq:table} we have
\[
(k-1)p+1 \equiv 0 \pmod{12}.
\]
Consider the $4$ possible classes of $(k,p)\pmod{12}$. If $(k,p) \equiv (2,11) \pmod{12}$,
then we have $\ep_{2}(pN)=\ep_{3}(pN)=0$, so \eqref{eqn:main ineq} reduces to \eqref{eqn:simple1}. 
Here, we have $k \geq 14$, so \eqref{eqn:simple1} holds. If  $(k,p) \equiv (6,7)\pmod{12}$, then \eqref{eqn:main ineq} becomes \eqref{eqn:2}, which holds because $k \geq 6$ and $\ep_{3}(N) \geq 1$. 
If $(k,p) \equiv (8,5) \pmod{12}$, then \eqref{eqn:main ineq} becomes \eqref{eqn:simple3}. We have $k \geq 8$ and $\ep_{2}(N) \geq 1$, so \eqref{eqn:simple3} follows. 
Finally, if $(k,p) \equiv (0,1) \pmod{12}$, then \eqref{eqn:main ineq} becomes
\begin{equation}
\mfrac{k-2}{12}I(N)+\mfrac{1}{2}\ep_{2}(N)+\mfrac{2}{3}\ep_{3}(N) \geq 1,
\end{equation}
which holds since $\ep_{2}(N) \geq 1$ and $\ep_{3}(N) \geq 1$. This finishes the proof when $\alpha_{2}=\alpha_{3}=0$. The remaining cases use similar ideas; fewer details will be given.
\subsection{$\alpha_{2} \neq 0$ and $\alpha_{3} = 0$}
In this case, \eqref{eqn:main ineq} becomes
\begin{equation}\label{eqn:simple5}
\mfrac{k-2}{12}I(N) +\(\floor{\mfrac{k}{4}}  - \mfrac{k-1}{4}\)\ep_{2}(pN)+\(\floor{\mfrac{k}{3}}-\mfrac{k-1}{3}\)\ep_{3}(pN)+\mfrac{1}{2}\alpha_{2} \geq 1.
\end{equation}
    
     If $\ep_{3}(N)=0$, then $N \geq 2$. Since $I(N) \geq 3$ and $k \geq 4$, \eqref{eqn:simple5} holds. So, assume that $\ep_{3}(N) \neq 0$. By \eqref{eq:table}, we have $(k-1)p+1 \equiv 6 \pmod{12}.$ The strategy is then to consider the $4$ possibilities for $(k,p)\pmod{12}$. We illustrate this only when $(k,p) \equiv (6,1) \pmod{12}$. In this case, the quantity in \eqref{eqn:simple5} is at least $\frac{1}{3}I(N)+\frac{2}{3}\ep_{3}(N)  \geq 1$.

\subsection{$\alpha_{2}=0$ and $\alpha_{3} \neq 0$} 
In this case, \eqref{eqn:main ineq} reduces to
\begin{equation}\label{eqn:simple6}
\mfrac{k-2}{12}I(N) +\(\floor{\mfrac{k}{4}}  - \mfrac{k-1}{4}\)\ep_{2}(pN)+\(\floor{\mfrac{k}{3}}-\mfrac{k-1}{3}\)\ep_{3}(pN)+\mfrac{1}{3}\alpha_{3} \geq 1.
\end{equation}
If $\ep_{2}(N)=0$, then $N \geq 3$, so \eqref{eqn:simple6} holds. So, assume that $\ep_{2}(N) \neq 0$. By \eqref{eq:table}, we have 
\[
(k-1)p+1 \equiv 8 \pmod{12},
\]
so that $\alpha_{3} \geq 2$. We illustrate only the case $(k,p) \equiv (8,1) \pmod{12}$. In this case, \eqref{eqn:simple6} reduces to \eqref{eqn:simple3}, which holds since $k \geq 8$.
\subsection{$\alpha_{2} \neq 0 \text{ and } \alpha_{3} \neq 0 $} 
By \eqref{eq:table}, we have $(k-1)p+1 \equiv 2 \text{ or } 10 \pmod{12 }$. We illustrate only the case $(k,p) \equiv (2,1) \pmod{12}$. In this case, $\alpha_{3}=2\ep_{3}(N)$, so \eqref{eqn:main ineq} reduces to \eqref{eqn:simple1}, which holds since $k \geq 14.$

%For example, if $p$ is a prime with $p \equiv 11 \pmod{12}$, one does not have to consider cases involving $k \pmod{12}$ in view of the definitions of $\ep_{2}(N)$ and $\ep_{3}(N)$.
\section{Proofs of Corollary~\ref{thm:analogue} and Corollary~\ref{thm:Subspace}}
\begin{proof}[Proof of Corollary~\ref{thm:analogue}]
Suppose that $N$ is a positive integer, that $k$ is a positive even integer, and that $p$ is a prime with $p \geq \max(5,k+1)$ and $p \nmid N$. Since every non-zero $f \in S_{k}(N)$ has bounded denominators, we may assume that $v_p(f)=0$. Since $S_k(N)=\{0\}$, we have
\[
\Tr_{N}^{pN}(f\sk W_p^{pN})=f\sk W_p^{pN}+p^{1-\frac{k}{2}}f\sl U_p=0.
\] 
Thus, $v_p(f\sk W_p^{N}) \geq 1-\frac{k}{2}$, so Corollary~\ref{thm:analogue} follows from Theorem~\ref{thm:main}.     
\end{proof}
\begin{proof}[Proof of Corollary~\ref{thm:Subspace}]
Define the subspace
\[
S:=\{f\in S_k(pN): f\sk W_p^{pN}+p^{1-\tfrac{k}{2}}f\sl U_p=0\}.
\]     
Suppose that $0 \neq f \in S$. We apply Theorem~\ref{thm:main} after clearing denominators to conclude that $\ord_{\infty}(f) \leq~\dim(S_k(pN))$. Thus, $S\cap W_k(pN)=\{0\}$.
We also have $S \cong \ker(\Tr_{p}^{pN})$, since the Atkin-Lehner operator is an isomorphism. Since $\Tr_{p}^{pN}$ is surjective, we have
\[
\dim(S)=\dim(S_{k}(pN))-\dim(S_{k}(N)).
\] 
Since $S_{k}(pN)$ contains $S\oplus W_{k}(pN)$, we have $\dim(W_{k}(pN)) \leq \dim(S_k(N))$.
\end{proof}
\section{Acknowledgements}
The author would like to thank Scott Ahlgren for suggesting this project and making many helpful comments.
\bibliographystyle{amsalpha}    
      
\bibliography{Analogue} 
   
        \end{document}